\documentclass{amsart}

\usepackage{amsmath, amsthm, amscd, amsfonts, amssymb, graphicx, color}
\usepackage[bookmarksnumbered, colorlinks, plainpages]{hyperref}

\newcommand{\del}{\mbox{del}\,}

\newcommand{\lk}{\mbox{lk}\,}

\newcommand{\reg}{\mbox{reg}\,}

\newcommand{\pd}{\mbox{pd}\,}

\newcommand{\T}{\mathrm}

\newtheorem{theorem}{Theorem}[section]
\newtheorem{corollary}[theorem]{Corollary}
\newtheorem{lemma}[theorem]{Lemma}
\newtheorem{proposition}[theorem]{Proposition}
\newtheorem{definition}[theorem]{Definition}

\numberwithin{equation}{section}

\begin{document}
\bibliographystyle{amsplain}

\title[codismantlability and Projective dimension]{codismantlability and projective dimension of the Stanley-Reisner ring of special hypergraphs}
\author[F. Khosh-Ahang and S. Moradi]{Fahimeh Khosh-Ahang and Somayeh Moradi}
\address{Fahimeh Khosh-Ahang, Department of Mathematics,
 Ilam University, P.O.Box 69315-516, Ilam, Iran.}
\email{fahime$_{-}$khosh@yahoo.com}
\address{Somayeh Moradi, Department of Mathematics,
 Ilam University, P.O.Box 69315-516, Ilam, Iran and School of Mathematics, Institute
 for Research in Fundamental Sciences (IPM), P.O.Box: 19395-5746, Tehran, Iran.} \email{somayeh.moradi1@gmail.com}

\begin{abstract}

\noindent  In this paper firstly, we generalize the concept of codismantlable graphs to hypergraphs and show that some special vertex decomposable hypergraphs are codismantlable. Then we generalize the concept of bouquet in graphs to hypergraphs to extend some combinatorial invariants of graphs about disjointness of a set of bouquets. We use these invariants to characterize the projective dimension of Stanley-Reisner ring of special hypergraphs in some sense.
\end{abstract}

\maketitle
\section*{Introduction}

Recently, in \cite{BC}, the authors introduced a new class of graphs in terms of a codominated vertex,  which they call codismantlable and studied some algebraic and combinatorial properties of
these graphs. They proved that a $(C_4,C_5)$-free vertex decomposable graph is codismantlable. In this paper,
we extend the concept of codismantlability to hypergraphs and gain some generalizations of the results for graphs in this context.

Corresponding to each hypergraph $\mathcal{H}$, a simplicial complex, called the \textbf{independence complex} of $\mathcal{H}$, is assigned. The independence complex of $\mathcal{H}$, which is denoted by  $\Delta_{\mathcal{H}}$, is one whose faces are the independent sets of vertices of $\mathcal{H}$, i.e. the sets which do not contain any edge of $\mathcal{H}$. Squarefree monomial ideals can be studied using these combinatorial concepts. Since edge ideal of graphs, as the easiest class of squarefree monomial ideals, are extensively studied (cf.  \cite{HD}, \cite{Kimura}, \cite{KM}, \cite{Mor}, \cite{VT} and \cite{Zheng}), many researchers have tried to extend the concepts in graphs to hypergraphs and to find more general results in hypergraphs (cf. \cite{Emt}, \cite{HT1}, \cite{HW},  \cite{MVi} and  \cite{Wood}). Explaining the algebraic properties of the Stanley-Reisner ring of $\mathcal{H}$ (such as projective dimension) in terms of the graphic invariants of $\mathcal{H}$ is the main purpose of many works \cite{DS,DHS,FV,LM,WW}.

The \textbf{Castelnuovo-Mumford regularity} (or simply regularity) and the projective dimension
of an $R$-module $M$ are respectively defined as
$$\reg(M) := \max\{j-i |\
\beta_{i,j}(M)\neq 0\},$$
and
$$\pd(M) := \max\{i |\
\beta_{i,j}(M)\neq 0 \ \T{for \ some \ } j\},$$
where $\beta_{i,j}$ is the $(i,j)$-th Betti number of $M$.

In this paper, we also study the projective dimension and regularity of the Stanley-Reisner ring of $\Delta_{\mathcal{H}}$ for some families of hypergraphs and relate them to some combinatorial concepts. Also, we generalize some results, which had been gained for graphs, such as Propositions 2.6 and 2.7 and Corollary 2.9 in \cite{KM}.

The paper proceeds as follows. In Section 1, we study codismantlable hypergraphs.
In fact, in Definition \ref{codis}, we generalize codominated vertex concept to hypergraphs. We show that under some circumstances a vertex $x$ of a hypergraph is a codominated vertex if and only if it is a shedding vertex (see Theorem \ref{main}). By using this fact, in Corollary \ref{coro}, we conclude that
a $C_5$-free vertex decomposable hypergraph, which all of its $3$-cycles have edges of cardinality just two, is codismantlable.
Hence, we generalize Lemma 2.6 in \cite{BC} and Lemma 2.3 in \cite{KM}.

In Section 2, we study the projective dimension of Stanley-Reisner ring of certain hypergraphs. To this aim, by generalizing the concept of bouquet from graphs to hypergraphs, we introduce some invariants such as $d_\mathcal{H}$ and $d'_\mathcal{H}$ for a hypergraph $\mathcal{H}$. We show that $c_\mathcal{H}\leq d_\mathcal{H}\leq d'_\mathcal{H}$ and for vertex decomposable hypergraph $\mathcal{H}$, $d'_\mathcal{H}$ is an upper bound for projective dimension of the Stanley-Reisner ring of $\mathcal{H}$. Also, we find special circumstances in Theorem \ref{final} under which, $d'_\mathcal{H}$ precisely characterizes $\T{bigheight}(I(\mathcal{H}))$ and projective dimension of the Stanley-Reisner ring of $\mathcal{H}$. In fact in this section we generalize Propositions 2.6, 2.7 and Corollary 2.9 in \cite{KM} in some sense.

\section{Codominated vertex, shedding vertex and codismantlable hypergraphs}

Throughout this paper, we assume that $\mathcal{H}=(V(\mathcal{H}), \mathcal{E}(\mathcal{H}))$ is a simple hypergraph with vertex set $V(\mathcal{H})$ and edge set $\mathcal{E}(\mathcal{H})$ whose elements are a family of subsets of $V(\mathcal{H})$ such that no element of $\mathcal{E}(\mathcal{H})$ contains another. Also, for any vertex $x\in V(\mathcal{H})$, $\mathcal{H}\setminus x$ is a hypergraph with vertex set $V(\mathcal{H})\setminus\{x\}$ and  edge set $\{E\in \mathcal{E}(\mathcal{H}):\ x\notin E\}$.
Moreover $\mathcal{H}/x$ is a hypergraph with vertex set  $V(\mathcal{H})\setminus\{x\}$ whose edges are the minimal elements (with respect to inclusion) of the set $\{E\setminus \{x\} \ : \ E\in \mathcal{E}(\mathcal{H})\}$. It is clear that $\mathcal{H}\setminus x$ and $\mathcal{H}/x$ are two simple hypergraphs.
They are called \textbf{deletion} and \textbf{contraction} of $\mathcal{H}$ by $x$, respectively. Note that for a vertex $x\in V(\mathcal{H})$, $\del_{\Delta_\mathcal{H}}(x)=\Delta_{\mathcal{H}\setminus x}$ and $\lk_{\Delta_{\mathcal{H}}}(x)=\Delta_{\mathcal{H}/x}$. Furthermore, we assume that $R=K[x_1, \dots, x_n]$, where $K$ is a field and $V(\mathcal{H})=\{x_1, \dots, x_n\}$. For more details in the context of hypergraphs we refer the reader to \cite{Berge}.

\begin{definition}
{\em Let $\Delta$ be a simplicial complex on the vertex set $V =
\{x_1,\ldots, x_n\}$. Then $\Delta$ is \textbf{vertex decomposable}
if either:

1) The only facet of $\Delta$ is $\{x_1,\ldots, x_n\}$, or
$\Delta=\emptyset$.

2) There exists a vertex $x\in V$ such that $\del_{\Delta}(x)$ and
$\lk_{\Delta}(x)$ are vertex decomposable, and such that every facet
of $\del_{\Delta}(x)$ is a facet of $\Delta$.}
\end{definition}

A vertex $x\in V$ for which every facet
of $\del_{\Delta}(x)$ is a facet of $\Delta$ is called a
\textbf{shedding vertex} of $\Delta$. Note that this is equivalent to say that no facet of $\lk_{\Delta}(x)$ is a facet of
$\del_{\Delta}(x)$.

A hypergraph $\mathcal{H}$ is called vertex decomposable, if the independence
complex $\Delta_{\mathcal{H}}$ is vertex decomposable and a vertex of $\mathcal{H}$ is called a shedding vertex if it is a shedding vertex of $\Delta_{\mathcal{H}}$. It is easily seen that if $x$ is a shedding vertex of $\mathcal{H}$ and $\{E_1, \dots, E_k\}$ is the set of all edges of $\mathcal{H}$ containing $x$, then every facet of $\mathcal{H}\setminus x$ contains $E_i\setminus \{x\}$ for some $1\leq i\leq k$.

Also, recall that a vertex $x$ of a graph $G$ is called codominated if there is a vertex $y$ in $G$ such that $N_G[y]\subseteq N_G[x]$, where $N_G[x]$ is the set of vertices of $G$ consisting $x$ and all neighbors of $x$.

Now, we extend the notion of codominated vertex to hypergraphs.
\begin{definition}\label{codis}
Let $x,y$ be two distinct vertices of the hypergraph $\mathcal{H}$. We set
 $$N_{\mathcal{H}}(x\setminus y)=\{E\setminus \{x\} \ : \ E\in \mathcal{E}(\mathcal{H}), x\in E, y\not\in E\}.$$
Also, assume that $x\in V(\mathcal{H})$ and $\{E_1, \dots, E_k\}$ is the set of all edges of $\mathcal{H}$ containing $x$. Then the vertex $x$ is called a \textbf{codominated} vertex if there is an integer $1\leq i \leq k$ such that for each vertex $y\in E_i\setminus \{x\}$ we have
$$N_{\mathcal{H}}(y\setminus x)\subseteq N_{\mathcal{H}}(x\setminus y).$$
\end{definition}

Note that this definition is a natural generalization of one in graph theory. Now, we are going to find relations between two concepts of shedding vertex and codominated vertex. The following lemma shows that in general every codominated vertex is a shedding vertex.

\begin{lemma}(Compare \cite[Lemma 2.6]{BC}.)\label{codominated}
In a hypergraph $\mathcal{H}$, every codominated vertex is a shedding vertex.
\end{lemma}
\begin{proof}
Assume that $x$ is a codominated vertex and $\{E_1, \dots, E_k\}$ is the set of all edges of $\mathcal{H}$ containing $x$. Then there is an integer $1\leq i \leq k$ such that for each vertex $y\in E_i\setminus \{x\}$ we have
$$N_{\mathcal{H}}(y\setminus x)\subseteq N_{\mathcal{H}}(x\setminus y).$$
Assume that $S$ is a facet of $\Delta_{\mathcal{H}/x}$. Then there is a vertex $y\in (E_i\setminus\{x\})\setminus S$. Now, if we show that $S\cup \{y\}$ is a face of $\Delta_{\mathcal{H}\setminus x}$, the proof will be done. To this end, suppose in contrary that there is an edge $E'$ of $\mathcal{H}$ such that $x\not\in E'$ and $E'$ is contained in $S\cup \{y\}$. Since $S$ is a facet of $\Delta_{\mathcal{H}/x}$, we should have $E'\setminus\{y\}\in N_{\mathcal{H}}(y\setminus x)$. So $E'\setminus\{y\}\in N_{\mathcal{H}}(x\setminus y)$ which means that $(E'\setminus\{y\})\cup\{x\}\in \mathcal{E}(\mathcal{H})$. Therefore $E'\setminus\{y\}\in \mathcal{E}(\mathcal{H}/x)$. But $E'\setminus\{y\}\subseteq S$ which is a contradiction. Hence $S\cup \{y\}$ is a face of $\Delta_{\mathcal{H}\setminus x}$ as required.
\end{proof}

In the following definition we recall the concept of a cycle hypergraph from \cite{Berge}.
\begin{definition}
A closed chain is called a \textbf{cycle}. More precisely, assume that $\mathcal{H}$ is a hypergraph, $x_1, \dots, x_n$ are distinct vertices of $\mathcal{H}$ and $E_1, \dots, E_n$ are distinct edges of $\mathcal{H}$ such that $x_1, x_n\in E_n$ and $x_i,x_{i+1}\in E_i$, for each $1\leq i\leq n-1$. Then we  call $x_1-E_1-x_2-E_2-\dots -x_n-E_n-x_1$ or briefly $E_1-\dots -E_n$ a cycle of length $n$ or an $n$-cycle and we denote it by $C_n$. We say that $\mathcal{H}$ is $C_n$-free if it doesn't contain any cycle $C_n$ as a subhypergraph.
\end{definition}

In \cite[Lemma 2.3]{KM}, it is shown that in a $C_5$-free graph, any shedding vertex is codominated. In the following we generalize this result to hypergraphs.

\begin{lemma}\label{shedding}
Assume that $\mathcal{H}$ is a $C_5$-free hypergraph which all of its $3$-cycles have edges of cardinality just two. Then every shedding vertex is a codominated vertex.
\end{lemma}
\begin{proof}
Assume that $x$ is a shedding vertex and  $\{E_1, \dots, E_k\}$ is the set of all edges containing $x$. Moreover, assume, in contrary, that for each $1\leq i \leq k$, there is $y_{i}\in E_i\setminus\{x\}$ and an edge $E_i'$ of $\mathcal{H}$ such that $y_i\in E_i'$, $x\not\in E_i'$ and $E_i'\setminus \{y_i\}\in N_{\mathcal{H}}(y_i\setminus x)\setminus N_{\mathcal{H}}(x\setminus y_i)$. So, $(E_i'\setminus \{y_i\})\cup \{x\}$ is not an edge of $\mathcal{H}$. Note that for two distinct integers $1\leq i,j\leq k$, if $y_i=y_j$, we can choose $E_i'=E_j'$. Also, if $y_i\neq y_j$, then we should have $\{y_i,y_j\}\not\subseteq E_i'$ and $\{y_i,y_j\}\not\subseteq E_j'$. Because otherwise, if for instance $\{y_i,y_j\}\subseteq E_i'$, then  $x - E_i - y_i -E_i' - y_j-E_j - x$ is a cycle which $|E_i'|>2$ or $|E_j|>2$ which contradicts with our assumption (note that if $|E_i'|=|E_j|=2$, then we should have $E_i'=\{y_i, y_j\}$ and  $E_j=\{x,y_j\}$ and hence, ($E_i'\setminus \{y_i\})\cup \{x\}=E_j$, which is an edge of $\mathcal{H}$ and contradicts with the choice of $E_i'$). Set $S=\bigcup_{i=1}^k (E_i'\setminus \{y_i\})$. Since $\mathcal{H}$ is $C_5$-free, $S$ is an independent set of vertices in $\mathcal{H}/x$. To be more precise, if $S$ is not an independent set of vertices in $\mathcal{H}/x$, there is an edge $E\setminus \{x\}$ of $\mathcal{H}/x$ contained in $S$. So, there are distinct edges $E_i'$ and $E_j'$ with $1\leq i,j \leq k$ and distinct vertices $z,w\in E\setminus \{x\}$ such that $z\in E\cap (E_i'\setminus\{y_i\})$ and $w\in E\cap (E_j'\setminus\{y_i\})$ (otherwise $E\setminus \{x\}\subseteq E_i'\setminus\{y_i\}$ for some $1\leq i\leq k$. If $x\notin E$, then $E=E\setminus \{x\}\subseteq E_i'\setminus\{y_i\} \subset E'_i$, which is impossible. If $x\in E$, then $E=E_j$ for some $1\leq j\leq k$ and
 $y_j\in E_j\setminus \{x\}=E\setminus \{x\}\subseteq E'_i\setminus \{y_i\}$. Thus $i\neq j$ and  $\{y_i,y_j\}\subseteq E'_i$, which is a contradiction). Therefore, one can easily check that
$w - E - z - E_i' - y_i - E_i - x - E_j - y_j - E_j' - w$ forms a cycle of length five in $\mathcal{H}$ which is a contradiction. Now, we extend $S$ to a facet $F$ of $\Delta _{\mathcal{H}/x}$. $F$ is a face of $\Delta _{\mathcal{H}\setminus x}$ and so it is contained in a facet $G$ of $\Delta _{\mathcal{H}\setminus x}$. Now, since $x$ is a shedding vertex, $G$ should contain $E_i\setminus \{x\}$ for some $1\leq i\leq k$. Hence, $E'_i\subseteq G$ which contradicts to the fact that $G$ is a facet of $\Delta _{\mathcal{H}\setminus x}$.
\end{proof}

Now, by means of Lemmas \ref{codominated} and \ref{shedding}, we have the following result, which is one of our main results of this paper.
\begin{theorem}\label{main}
Assume that $\mathcal{H}$ is a $C_5$-free hypergraph which all of its $3$-cycles have edges of cardinality just two. Then a vertex $x$ of $\mathcal{H}$ is a shedding vertex if and only if it is a codominated vertex.
\end{theorem}

\begin{corollary}(Compare \cite[Theorem 2.7]{BC} and \cite[Lemma 2.3]{KM}.)
In a $C_5$-free graph, a vertex is a shedding vertex if and only if it is a codominated vertex.
\end{corollary}

In the following, we generalize the concept of codismantlability to hypergraphs.
\begin{definition}(Compare \cite[Definition 2.2]{BC}.)
Given two hypergraphs $\mathcal{G}$ and $\mathcal{H}$, we say that $\mathcal{G}$ is codismantlable to $\mathcal{H}$ if there exist hypergraphs $\mathcal{H}_0, \mathcal{H}_1, \dots, \mathcal{H}_{k+1}$ satisfying $\mathcal{G}\cong \mathcal{H}_0$, $\mathcal{H}\cong \mathcal{H}_{k+1}$ and $\mathcal{H}_{i+1}=\mathcal{H}_i\setminus x_i$, for each $0\leq i \leq k$, where $x_i$ is a codominated vertex of $\mathcal{H}_i$. A hypergraph $\mathcal{H}$ is called \textbf{codismantlable} if either it is an edgeless hypergraph or it is codismantlable to an edgeless hypergraph.
\end{definition}

The next corollary, which is a generalization of Corollary 2.9 in \cite{BC}, illustrates that certain vertex decomposable hypergraphs are codismantlable.

\begin{corollary} \label{coro}
Assume that $\mathcal{H}$ is a $C_5$-free vertex decomposable hypergraph which all of its $3$-cycles have edges of cardinality just two. Then $\mathcal{H}$ is codismantlable.
\end{corollary}
\begin{proof}
The result can be easily gained by an induction and Theorem \ref{main}.
\end{proof}

So, in graph theory, we have the following remarkable result.
\begin{corollary}(Compare \cite[Corollary 2.9]{BC}.)
Every $C_5$-free vertex decomposable graph is codismantlable.
\end{corollary}

\section{Projective dimension of edge ideal of certain hypergraphs}
In this section, we are going to characterize or even find some bounds for the projective dimension of edge ideal of special hypergraphs up to our ability. To this end, firstly we generalize some concepts from graphs to hypergraphs so that we generalize some combinatorial invariants of hypergraphs as follows.

\begin{definition}
A hypergraph $\mathcal{H}$ is called a \textbf{bouquet} if $\bigcap_{E\in \mathcal{E}(\mathcal{H})}E\neq \emptyset$. In this case, if $\mathcal{H}$ has at least two edges, then all elements in $\bigcap_{E\in \mathcal{E}(\mathcal{H})}E$ are called the \textbf{roots} of $\mathcal{H}$, all of its edges are called the \textbf{stems} of $\mathcal{H}$ and the elements of $\bigcup_{E\in \mathcal{E}(\mathcal{H})}(E\setminus \bigcap_{E\in \mathcal{E}(\mathcal{H})}E)$ are called the \textbf{flowers} of $\mathcal{H}$. When $\mathcal{H}$ has only one edge $E$, any proper subset of $E$ can be considered as roots and its complement as flowers of the bouquet. A subhypergraph of a simple hypergraph $\mathcal{T}$ which is a bouquet is called a bouquet of $\mathcal{T}$. Let
$\mathcal{B}=\{B_1,\ldots,B_n\}$ be a set of bouquets of $\mathcal{H}$. We
use the following notations.
$$\ \ \ F(\mathcal{B})=\{w\in V(\mathcal{H})\ | \ w \text{ is a flower of some bouquet in } \mathcal{B}\}$$
$$R(\mathcal{B})=\{z\in V(\mathcal{H}) \ | \ z \text{ is a root of some bouquet in } \mathcal{B}\}$$
$$S(\mathcal{B})=\{E\in E(\mathcal{H}) \ | \ E \text{ is a stem of some bouquet in } \mathcal{B}\}$$
\end{definition}

Kimura in \cite{Kimura} introduced two notions of disjointness
of a set of bouquets in graphs. In the following, we generalize these notions to hypergraphs.

\begin{definition}(See \cite[Definitions 2.1 and 5.1]{Kimura}.)
A set of bouquets $\mathcal{B}=\{B_1,\ldots,B_n\}$ is called
\textbf{strongly disjoint}  in $\mathcal{H}$ if we can choose a stem $E_i$ from each bouquet $B_i\in
\mathcal{B}$ such that $\{E_1,\ldots,E_n\}$ is an induced matching in $\mathcal{H}$.

A set of bouquets $\mathcal{B}=\{B_1,\ldots,B_n\}$ is called
\textbf{semi-strongly disjoint}  in $\mathcal{H}$ if
$R(\mathcal{B})$ is an independent set of vertices in $\mathcal{H}$.

Now, set $$d_\mathcal{H}:=\max\{|F(\mathcal{B})| \ | \ \mathcal{B}\
\text{is a strongly disjoint set of bouquets of } \mathcal{H} \}$$ and
$$\ \ \ \ \ \ d'_\mathcal{H}:=\max\{|F(\mathcal{B})| \ | \ \mathcal{B}\ \text{is a semi-strongly disjoint set of bouquets of } \mathcal{H} \}.$$
\end{definition}

Note that when $G$ is a graph, to define $d_G$ and $d'_G$, the condition $(i)$ in Definitions 2.1 and 5.1 in \cite{Kimura} is redundant. Therefore, the above definitions are suitable generalizations to hypergraphs and when $\mathcal{H}$ is a graph, they are coincide to those in graphs defined by Kimura in \cite{Kimura}.

Now, we are going to compare some graphical invariants together. So, firstly we recall the following definition and theorem.

\begin{definition}(See \cite[Definition 1.1]{KM1}.)
A set $\{E_1, \dots, E_k\}$ of edges of a hypergraph $\mathcal{H}$ is called a \textbf{semi induced matching} if the only edges contained in $\bigcup_{\ell=1}^kE_\ell$ are $E_1, \dots, E_k$. A semi induced matching which all of its elements are mutually disjoint is called an \textbf{induced matching}. Also, we set
$$c_{\mathcal{H}}:=\max\{|\bigcup_{\ell=1}^kE_\ell|-k \ : \ \{E_1, \dots, E_k\} \ \T{is \ an \ induced \ matching \ in \ }\mathcal{H}\},$$
$$c'_{\mathcal{H}}:=\max\{|\bigcup_{\ell=1}^kE_\ell|-k \ : \ \{E_1, \dots, E_k\} \ \T{is \ a \ semi \ induced \ matching \ in \ }\mathcal{H}\},$$
and we call them \textbf{induced matching number} and \textbf{semi induced matching number} of $\mathcal{H}$, respectively.
\end{definition}

\begin{theorem}(See \cite[Theorems 1.4 and 2.6]{KM1}\label{graph}
\begin{itemize}
\item[(i)] If $G$ is a simple graph, then we have $c_G=c'_G$.
\item[(ii)] Let $\mathcal{H}$ be a $(C_2,C_5)$-free vertex decomposable hypergraph.Then
$$\T{reg}(R/I_{\Delta_{\mathcal{H}}})\leq c'_{\mathcal{H}}\leq \T{dim}(\Delta_{\mathcal{H}})+1.$$
\end{itemize}
\end{theorem}

Now, the following result shows that $d_\mathcal{H}$ and $d'_\mathcal{H}$ are comparable with $c_\mathcal{H}$ and $c'_\mathcal{H}$.
\begin{proposition}\label{cd}
\begin{itemize}
\item[(i)] For each hypergraph $\mathcal{H}$, we have
$$c_\mathcal{H}\leq d_\mathcal{H}\leq d'_\mathcal{H}.$$
Specially when $G$ is a graph, we have
$$c_G=c'_G\leq d_G\leq d'_G.$$
\item[(ii)] If $\mathcal{H}$ is a $C_2$-free hypergraph, then
$$c_\mathcal{H}\leq c'_\mathcal{H}\leq d'_\mathcal{H}.$$
\end{itemize}
\end{proposition}
\begin{proof}
\begin{itemize}
\item[(i)] It is easy to see that if $\{E_1, \dots, E_k\}$ is an induced matching in $\mathcal{H}$, then one can
consider it as a set of bouquets with only one stem and one vertex as their roots. So, the number of its flowers equals to $c_\mathcal{H}$. Also, one may see that every strongly disjoint set of bouquets in $\mathcal{H}$ is semi-strongly disjoint. Hence, we have $c_\mathcal{H}\leq d_\mathcal{H}\leq d'_\mathcal{H}$ as desired.
The last assertion immediately follows from Theorem \ref{graph}(i).
\item[(ii)] Since every induced matching of $\mathcal{H}$ is a semi induced matching, it is enough to prove the second inequality. In this regard,
assume that $\{E_1, \dots, E_k\}$ is a semi-induced matching in $\mathcal{H}$ such that $c'_\mathcal{H}=|\bigcup_{\ell=1}^kE_\ell |-k$. At first, remove all edges $E_j$ which is contained in $\bigcup_{\ell=1,\ell\neq j}^kE_\ell$. Then set $s=0$ and assume that $B_0$ is a bouquet in $\mathcal{H}$ with $\mathcal{E}(B_0)=\emptyset$. Then for $i =1, \dots, k$, if $E_i \cap (\bigcap_{E_\ell\in \mathcal{E}(B_j)}E_\ell)\neq\emptyset$ for some $j<s+1$, then set $\mathcal{E}(B_j):=\mathcal{E}(B_j)\cup \{E_i\}$ (note that if there exist more than one $j<s+1$ such that $E_i \cap (\bigcap_{E_\ell\in \mathcal{E}(B_j)}E_\ell)\neq\emptyset$, then we add $E_i$ to edges of just one of these $B_j$s); else consider $B_{s+1}$ as a bouquet with $\mathcal{E}(B_{s+1})=\{E_i\}$ and set $s:=s+1$. Now, since $\{E_1, \dots, E_k\}$ is a semi-induced matching in $\mathcal{H}$, $\mathcal{B}=\{B_1, \dots, B_s\}$ is a semi-strongly disjoint set of bouquets in $\mathcal{H}$. Also, we have
$$|F(\mathcal{B})|\geq |(\bigcup_{\ell=1}^kE_\ell)\setminus R(\mathcal{B})|=|\bigcup_{\ell=1}^kE_\ell|-|R(\mathcal{B})|.$$
Now, since $\mathcal{H}$ is $C_2$-free, by considering only one vertex as the root of bouquets with one stem, we have $|R(\mathcal{B})|\leq k$. Hence, $$|F(\mathcal{B})|\geq |\bigcup_{\ell=1}^kE_\ell|- k=c'_\mathcal{H},$$ which implies that $c'_\mathcal{H}\leq d'_\mathcal{H}$ as required.
\end{itemize}
\end{proof}

The following corollary is an immediate consequence of part (ii) of Theorem \ref{graph} and part (ii) of Proposition \ref{cd}.
\begin{corollary}
Assume that $\mathcal{H}$ is a $(C_2,C_5)$-free vertex decomposable hypergraph. Then
$$c_\mathcal{H}\leq \T{reg}(R/I_{\Delta_\mathcal{H}})\leq c'_\mathcal{H}\leq d'_\mathcal{H}.$$
So, if moreover $c_\mathcal{H}=d'_\mathcal{H}$, then
$$\T{reg}(R/I_{\Delta_\mathcal{H}})=c_\mathcal{H}=d'_\mathcal{H}.$$
\end{corollary}

For our main result of this section, we need the following two lemmas.
\begin{lemma}\label{1}
Assume that $\mathcal {H}$ is a hypergraph and $x$ is a shedding vertex of $\mathcal{H}$. Then $d'_{\mathcal{H}/x}\leq d'_\mathcal{H}$.
\end{lemma}
\begin{proof}
Suppose that $\mathcal{B}=\{B_1,\ldots,B_n\}$ is a semi-strongly disjoint set of bouquets of $\mathcal{H}/x$ such that $d'_{\mathcal{H}/x}=|F(\mathcal{B})|$ and for each $1\leq i \leq n$, $S(B_i)=\{E_{i,1}\setminus \{x\}, \dots, E_{i,m_i}\setminus \{x\}\}$. Then if for each $1\leq i \leq n$ we set $B'_i$ as a bouquet in $\mathcal{H}$ with $S(B'_i)=\{E_{i,1}, \dots, E_{i,m_i}\}$, then clearly $\mathcal{B'}=\{B'_1,\ldots,B'_n\}$ is a semi-strongly disjoint set of bouquets of $\mathcal{H}$ with $|F(\mathcal{B'})|\geq d'_{\mathcal{H}/x}$. Hence, $d'_{\mathcal{H}/x}\leq d'_\mathcal{H}$ as required.
\end{proof}

\begin{lemma}\label{2}
Assume that $\mathcal {H}$ is a hypergraph and $x$ is a shedding vertex of $\mathcal{H}$. Then $d'_{\mathcal{H}\setminus x}+1\leq d'_\mathcal{H}$.
\end{lemma}
\begin{proof}
Suppose that $\mathcal{B}=\{B_1,\ldots,B_n\}$ is a semi-strongly disjoint set of bouquets of $\mathcal{H}\setminus x$ such that $d'_{\mathcal{H}\setminus x}=|F(\mathcal{B})|$. Also, suppose that $\{E_1,\ldots,E_k\}$ is the set of all edges of $\mathcal{H}$ containing $x$. Now, we are going to prove that there is an integer $1\leq i\leq k$ such that for each $y_i\in E_i\setminus \{x\}$, $R(\mathcal{B})\cup \{y_i\}$ is an independent set of vertices in $\mathcal{H}$. If we prove this claim, then  by choosing this $E_i$ as a bouquet of $\mathcal{H}$ with one stem $E_i$ and some $y_i\in E_i\setminus \{x\}$ as its root, $\mathcal{B'}=\mathcal{B}\cup \{E_i\}$ forms a semi-strongly disjoint set of bouquets of $\mathcal{H}$ with $|F(\mathcal{B'})|\geq d'_{H\setminus x}+1$, since $x\in F(E_i)\setminus V(\mathcal{B})$. This completes the proof. To prove the claim suppose in contrary that for each $1\leq i\leq k$, there is a vertex $y_i\in E_i\setminus \{x\}$ and an edge $E'_i$ of $\mathcal{H}$ containing $y_i$ such that $E'_i\setminus \{y_i\}\subseteq R(\mathcal{B})$. Note that if for two distinct integers $1\leq i,j\leq k$, $y_i=y_j$, then one can choose $E'_i=E'_j$. Also, if $y_i\neq y_j$, then we should have $\{y_i,y_j\}\not\subseteq E_i'$ and $\{y_i,y_j\}\not\subseteq E_j'$. Because otherwise, if for instance $\{y_i,y_j\}\subseteq E_i'$, then $y_j\in E_i'\setminus \{y_i\}\subseteq R(\mathcal{B})$. Hence $E'_j\setminus \{y_j\}\subseteq R(\mathcal{B})$ insures that $E'_j\subseteq R(\mathcal{B})$ which contradicts to independence of $R(\mathcal{B})$ in $\mathcal{H}\setminus x$.
Now, if we set $S=\bigcup_{i=1}^k (E_i'\setminus \{y_i\})$, then $S$ is an independent set of vertices in $\mathcal{H}/x$. To be more precise, if $S$ is not an independent set of vertices in $\mathcal{H}/x$, there is an edge $E\setminus \{x\}$ of $\mathcal{H}/x$ contained in $S$. If $x\in E$, then $E=E_j$ for some $1\leq j\leq k$. Hence, $y_j\in E\setminus \{x\}\subseteq S$, which implies that $\{y_i, y_j\}\subseteq E'_i$ for some integer $1\leq i\leq k$ with $i\neq j$ which is a contradiction. So, $E$ is an edge of $\mathcal{H}\setminus x$ which is impossible, because $E\subseteq S\subseteq R(\mathcal{B})$ and $R(\mathcal{B})$ is independent in $\mathcal{H}\setminus x$.
Now, we extend $S$ to a facet $F$ of $\Delta _{\mathcal{H}/x}$. $F$ is a face of $\Delta _{\mathcal{H}\setminus x}$ and so it is contained in a facet $G$ of $\Delta _{\mathcal{H}\setminus x}$. Now, since $x$ is a shedding vertex, $G$ should contain $E_i\setminus \{x\}$ for some $1\leq i\leq k$. Hence, $E'_i\subseteq G$ which contradicts the fact that $G$ is a facet of $\Delta _{\mathcal{H}\setminus x}$. This proves our claim and so completes the proof.
\end{proof}

Now, we are ready to state our main result of this section which is a generalization of Proposition 2.6 in \cite{KM}.
\begin{theorem}(Compare \cite[Proposition 2.6]{KM}.)\label{pd}
Assume that $\mathcal{H}$ is a vertex decomposable hypergraph. Then $$\T{pd}(R/I_{\Delta_\mathcal{H}})\leq d'_\mathcal{H}.$$
\end{theorem}
\begin{proof}
We proceed by induction on $|V(H)|$.  If $|V(\mathcal{H})|=2$, the result is clear. Suppose, inductively, that the result has been proved for smaller values of $|V(\mathcal{H})|$. Assume that $x$ is a shedding vertex of $\mathcal{H}$. Let $\Delta=\Delta_{\mathcal{H}}$, $\Delta_1=\Delta_{\mathcal{H}\setminus x}$ and $\Delta_2=\Delta_{\mathcal{H}/x}$. Then $\mathcal{H}\setminus x$ and $\mathcal{H}/x$ are vertex decomposable hypergraphs and no facet of $\Delta_2$ is a facet of $\Delta_1$.
 By inductive hypothesis we have
$$\T{pd}(R/I_{\Delta_1})\leq d'_{\mathcal{H}\setminus x} \ \T{and} \ \T{pd}(R/I_{\Delta_2})\leq d'_{\mathcal{H}/x}.$$
On the other hand, by Corollary 2.10 in \cite{MK}, we have the equality
$$\T{pd}(R/I_{\Delta})=\max\{\T{pd}(R/I_{\Delta_1})+1,\T{pd}(R/I_{\Delta_2})\}.$$
Hence
$$\T{pd}(R/I_{\Delta})\leq \max\{d'_{\mathcal{H}\setminus x}+1,d'_{\mathcal{H}/x}\}.$$
Now, the result immediately follows from Lemmas \ref{1} and \ref{2}.
\end{proof}

We end this paper by the following result, which is a generalization of Proposition 2.7 and Corollary 2.9 in \cite{KM} and can characterize the projective dimension of certain hypergraphs in special circumstances.
\begin{theorem} (Compare \cite[Proposition 2.7 and Corollary 2.9]{KM}.)\label{final}
Assume that $\mathcal{B}=\{B_1,\ldots,B_n\}$ is a semi-strongly disjoint set of bouquets of $\mathcal{H}$ such that $d'_{\mathcal{H}}=|F(\mathcal{B})|$. Then
\begin{itemize}
\item[(i)] there exists a minimal vertex cover of $\mathcal{H}$ contained in $F(\mathcal{B})$.
\item[(ii)] If moreover all edges of $S(\mathcal{B})$ has cardinality two (specially if $\mathcal{H}$ is a graph), then $F(\mathcal{B})$ is a minimal vertex cover of $\mathcal{H}$ and so $$c_\mathcal{H}\leq d_\mathcal{H}\leq d'_\mathcal{H}\leq\T{bigheight}(I(\mathcal{H}))\leq \T{pd}(R/I(\mathcal{H})).$$
\item[(iii)] If $H$ is a vertex decomposable hypergraph and all edges of $S(\mathcal{B})$ has cardinality two (specially if $\mathcal{H}$ is a graph), then $$\T{bigheight}(I(\mathcal{H}))=\T{pd}(R/I(\mathcal{H}))=d'_\mathcal{H}.$$
\end{itemize}
\end{theorem}
\begin{proof}
\begin{itemize}
\item[(i)] Suppose that $E(\mathcal{H})\setminus S(\mathcal{B})=\{E_1, \dots, E_k\}$. Then for each $1\leq i\leq k$, $E_i\cap F(\mathcal{B})\neq\emptyset$. Because otherwise, if $E_i\cap F(\mathcal{B})=\emptyset$, then by adding the edge $E_i$ to $\mathcal{B}$ (as an stem or a bouquet with one stem), one can get to a semi-strongly disjoint set of bouquets of $\mathcal{H}$ with more flowers than $d'_{\mathcal{H}}$ which is impossible. Now, let $S(\mathcal{B})=\{E_{k+1}, \dots, E_t\}$ and set $S_0=\emptyset$. For each $1\leq i \leq t$, if $E_i\cap S_{i-1}\neq \emptyset$, then set $S_i:=S_{i-1}$, else choose a vertex $x_i\in E_i\cap F(\mathcal{B})$ and set $S_i:=S_{i-1}\cup \{x_i\}$. It can be easily seen that $S_t$ is a minimal vertex cover of $\mathcal{H}$ contained in $F(\mathcal{B})$ as desired.
\item[(ii)] The first statement can be easily seen by (i). The inequalities can be gained by Proposition \ref{cd}, the first statement and \cite[Corollary 3.33]{MVi}.
\item[(iii)] follows from (ii) and Theorem \ref{pd}.
\end{itemize}

\end{proof}


\providecommand{\bysame}{\leavevmode\hbox
to3em{\hrulefill}\thinspace}

\end{document}